\documentclass[12pt]{amsart}
\usepackage{fontenc}
\usepackage[english]{babel}
\usepackage[swapnames,normal]{frontespizio}

\usepackage[font=small,format=default,indention=1em,textfont=it,up]{caption}
\usepackage{enumitem}
\usepackage{subfig}
\theoremstyle{plain}
\newtheorem{theorem}{Theorem}[section]
\newtheorem{lemma}[theorem]{Lemma}
\newtheorem{proposition}[theorem]{Proposition}
\newtheorem{corollary}[theorem]{Corollary}

\theoremstyle{remark}

\theoremstyle{remark}
\newtheorem{remark}[theorem]{Remark}
\theoremstyle{definition}
\newtheorem{definition}[theorem]{Definition}
\numberwithin{equation}{section}

\def\|{\parallel}
\def\B{\mathcal{B}_1}
\def\S{\mathcal{S}^1}
\def\R{\mathbf{R}}
\def\H{\mathcal{H}}
\def\E{\mathcal{E}}

\begin{document}
\title[Alternating and Variable Controls]
{Variable Support Control for the Wave Equation: A Multiplier Approach}

\author{Antonio Agresti}
\author{Daniele Andreucci}
\author{Paola Loreti}

\address{Dipartimento di Scienze di Base e Applicate per l'Ingegneria\\
Sapienza Universit\`a di Roma\\
via A.Scarpa 16, 00161 Roma, Italy}
\curraddr{Dipartimento di Matematica Guido Castelnuovo\\
Sapienza Universit\`a di Roma\\
P.le A. Moro 2, 00100 Roma, Italy}
\address{Dipartimento di Scienze di Base e Applicate per l'Ingegneria\\
Sapienza Universit\`a di Roma\\
via A.Scarpa 16, 00161 Roma, Italy}
\address{Dipartimento di Scienze di Base e Applicate per l'Ingegneria\\
Sapienza Universit\`a di Roma\\
via A.Scarpa 16, 00161 Roma, Italy}
\thanks{%AMS Classification: 35K55, 35K57, 35K65, 35B45.
The second author is member of Italian
G.N.F.M.-I.N.d.A.M.}

\begin{abstract}
We study the controllability of the multidimensional wave equation in a bounded domain with Dirichlet boundary condition, in which the support of the control is allowed to change over time.\\
The exact controllability is reduced to the proof of the observability inequality, which is proven by a multiplier method.
Besides our main results, we present some applications.
\end{abstract}

\maketitle

\section{INTRODUCTION}
\label{Introduction}
The controllability of the wave equation or more
generally  of  partial  differential  equations  has  been
studied intensively in the last 30 years. Exact controllability for evolutive systems is  a  challenging mathematical problem, also relevant in engineering applications.
The  exact  controllability  of  the  wave  equation  with Dirichlet boundary condition using the multiplier method is studied in
\cite{Lions} see also \cite{multiplier} for a systematic study of this method and, for an approach using  theory of semi-groups see \cite{Tucsnak}. All these results do not  allow  the  support  of  the  control  to  change  over time,  while  this  variability  is  required  in  some  applications (see \cite{Carcaterra2}).\\
In this paper we provide a controllability result which extends the classical controllability results and admits variability of the control support over time. We have to point out that even in the fixed support case, the subset of the boundary, on which the control acts, cannot be chosen arbitrarily; for an extensive discussion on these topics see \cite{Labeau-Rauch}.\\
The common strategy to prove the exact controllability is to study an equivalent property i.e. the exact observability for the adjoint system (for more on this see \cite{Fourier}, \cite{Tucsnak}). Our approach is based on the multiplier method (see \cite{multiplier}, \cite{Lions}), which seems the most powerful in the multidimensional case. Although the strategy of the proof is quite classical and follows essentially \cite{Lions}, this approach leads to some unexpected results and it opens some questions on the optimality of these results; see \cite{Preprint} for more on this.\\
Here we dwell more on the multidimensional case which seems to be closer to applications, providing some explicit examples in special geometries.\\ 
Let us begin with some notations:
\begin{itemize}
\item Let $\Omega$ be a bounded domain of $\mathbf{R}^d$ with $d\geq 1$, of class $C^2$ \textit{or} convex. From the hypothesis on the boundary, we know that the exterior normal vector $\nu$ is well defined $\mathcal{H}^{d-1}${-a.e.} on $\partial \Omega$; where $\mathcal{H}^{d-1}$ is the $d-1$ dimensional Hausdorff measure (see \cite{Evans}). Moreover, we denote $d\Gamma$ the measure $\mathcal{H}^{d-1}$ restricted to $\partial \Omega$.
\item For each $t\in(0,T)$ and $T>0$, we write $\Gamma(t)$ for an open subset of $\partial \Omega$.
%\item Let $N$ be a positive integer, let $\{t_i\}_{i=-1,\dots,N}$ be a family of positive numbers, such that $t_{-1}=0$ and set $T:=t_N$.
%\item Let $\{\Gamma_i\}_{i=0,\dots,N}$ be a family of open subsets of $\partial \Omega$.
\item Lastly, we define
\begin{equation}
\label{Sigma}
\Sigma := \bigcup_{t\in(0,T)}\Gamma(t)\times \{t\}\,;
\end{equation}
and we suppose it to be $\mathcal{H}^{d-1}\otimes \mathcal{L}^1$-measurable (as defined in \cite{Evans} Chapter 1).
\end{itemize}
In this paper we want to study the following property.
\begin{definition}[Exact controllability]
\label{exact controllability def}
We say that the system 
\begin{equation}
\begin{cases}
\label{control1}
w_{tt}- \Delta w =0\,,& (x,t)\in \Omega \times (0,T)\,,\\
w=0\,, &  (x,t)\in\partial \Omega\times(0,T) \setminus {\Sigma}\,,\\
w=v\,, & (x,t)\in\Sigma\,,\\
w=w_0\,, & (x,t) \in \Omega\times \{0\}\,,\\
%\label{control2}
w_t=w_1\,, &  (x,t) \in \Omega\times \{0\}\,;
\end{cases}
\end{equation}
is exactly controllable in time $T>0$, if for all $w_0,z_0\in L^2(\Omega)$ and $w_1,z_1\in H^{-1}(\Omega)$ there exists a control $v\in L^2\left( \Sigma\right)$ such that the unique solution $w\in C([0,T];L^2(\Omega))\cap C^1([0,T];H^{-1}(\Omega))$ of (\ref{control1}) satisfies
\begin{equation}
\label{eq:cond final state}
w(x,T)=z_0\,, \qquad w_t(x,T)=z_1\,.
\end{equation}
\end{definition}
Of course, the problem (\ref{control1}) has to be intended in a weak sense, which we give below (see Definition \ref{weak def}).\\ 
\\
%Some authors (e.g. \cite{Tucsnak}) call the property in Definition \ref{exact controllability def} null controllability and they reserves the name exact controllability for another property. Informally, the exact controllability means exists a $T>0$ such that for each $w_0,w_1,W_1,W_2$ there exists a control $v$ such that the state $w$ satisfies
%\begin{equation*}
%w(x,T)=W_0\,, \qquad  w_t(x,T)=W_1\,.
%\end{equation*}
%By time reversibility of the wave equation, the notion of null controllability and exact controllability coincides, so
In the following, it will be useful to know some properties of the solution of the wave equation with null Dirichlet boundary condition, i.e.
\begin{equation}
\label{omo}
\begin{cases}
u_{tt}- \Delta u =0\,,& (x,t)\in \Omega \times \mathbf{R}^+\,,\\
u=0\,, &  (x,t)\in\partial \Omega\times\mathbf{R}^+\,,\\
u=u_0\,, & (x,t) \in \Omega\times \{0\}\,,\\
%\label{control2}
u_t=u_1\,, &  (x,t) \in \Omega\times \{0\}\,;
\end{cases}
\end{equation} 
where $\mathbf{R}^+:=[0,\infty)$. In particular, we will need the following proposition; see \cite{Pazy} for the notion of classical and mild solution.
\begin{proposition}
\label{prop:well posedness omo}
The following holds true:
\begin{itemize}
\item[i)] For each $(u_0,u_1)\in H^1_0(\Omega)\times L^2(\Omega)$, then the problem (\ref{omo}) has an unique \textbf{mild} solution $u$ in the class
\begin{equation*}
C^1(\mathbf{R}^+;H^1_0(\Omega))\cap C(\mathbf{R}^+;L^2(\Omega))\,.
\end{equation*}
\item [ii)] For each $(u_0,u_1)\in (H^2(\Omega)\cap H^1_0(\Omega))\times H^1_0(\Omega)$, then the problem (\ref{omo}) has an unique \textbf{classical} solution $u$ in the class
\begin{equation*}
C(\mathbf{R}^+;H^2(\Omega)\cap H^1_0(\Omega))\cap C^2(\mathbf{R}^+;L^2(\Omega))\,.
\end{equation*}
\end{itemize}
\end{proposition}
\begin{proof}
Let $\{f_k\,:\,k\in \mathbf{N}\}$ and $\{\lambda_k\,:\,k\in  \mathbf{N}\}$ be respectively the eigenfunctions and the eigenvalues of the Laplace operator with null Dirichlet boundary condition. Then the solution $u$ of the problem (\ref{omo}) is
\begin{equation}
\label{eq:solution}
u(t)=
\sum_{k=0}^{\infty} \left( \cos(\lambda_k t)\hat{u}_{0,k} 
+ \frac{\sin(\lambda_j t)}{\lambda_k}  \hat{u}_{1,k}\right)  f_k\,;
\end{equation}
where
\begin{equation*}
\hat{u}_{i.k} := (u_i,f_k)_{L^2(\Omega)}:=\int_{\Omega} u_i\,f_k\, dx\,, \qquad i=0,1\,.
\end{equation*}
Then the representation of the solution in (\ref{eq:solution}) readily implies the claim in $i)-ii)$.
\end{proof}
With this in hand, we can give a definition of solution for the problem (\ref{control1}).
\begin{definition}[Weak solution of (\ref{control1})]
\label{weak def}
A map $w\in C([0,T];L^2(\Omega))\cap C^1([0,T];H^{-1}(\Omega))$ is a weak solution of (\ref{control1}) with $v\in L^2\left( \Sigma \right)$ and $(w_0,w_1)\in L^2(\Omega)\times H^{-1}(\Omega)$ if and only if
\begin{multline}
\label{weak formulation}
 - \langle   w_t(s),u(s)\rangle_{H^{-1}\times H^1_0} +
 \langle w(s),  u_t(s)\rangle_{L^2\times L^2}=\\
 - \langle   w_1,u_0\rangle_{H^{-1}\times H^1_0} +
 \langle w_0,  u_1\rangle_{L^2\times L^2} +\\
 \int_{0}^s \int_{\Gamma(t)} v\,  \partial_{\nu} u \,\,d\Gamma\,dt\,,
\end{multline}
for all $0<s\leq T$ and for all mild solutions $u$ of (\ref{omo}) and initial data $(u_0,u_1)\in H^1_0(\Omega)\times L^2(\Omega)$.
\end{definition}
For the motivation see \cite{Fourier}; for other equivalent formulations one may consult \cite{Lions}. 
\begin{remark}
\label{hidden}
Note that, on the RHS of (\ref{weak formulation}) it appears $\partial_{\nu} u$ i.e. the normal derivative of $u$. Although the regularity of a mild solution $u$ is not sufficient to have a well defined trace for $\partial_{\nu} u$, in the sense of Sobolev spaces (see \cite{Grisvard}), the normal derivative $\partial_{\nu} u$ is a well defined element of $L^2(\partial \Omega \times (0,t))$ for each $t>0$: this result is generally called \textit{hidden regularity}; the proof relies on a standard density argument, for details see \cite{multiplier}, \cite{Lions} or \cite{Hidden}. By this, the last term in the RHS of (\ref{weak formulation}) is well defined since $v\in L^2\left( \Sigma \right)$. 
\end{remark}
Before proceeding in the analysis of the weak solutions, we recall the following well known result about the solution of the wave equation (\ref{omo}).
\begin{corollary}[Energy conservation]
\label{energy cor}
For each $t\in\mathbf{R}^+$ and each mild solution $u$ of (\ref{omo}) with initial data $(u_0,u_1)\in H^1_0(\Omega)\times L^2(\Omega)$, the energy
\begin{equation}
\label{energy}
E(t)=\frac{1}{2} \left(\Vert \nabla u(\cdot,t) \Vert_{L^2(\Omega)}^2 + \Vert u_t(\cdot,t) \Vert_{L^2(\Omega)}^2\right)\,,
\end{equation}
is constant, and it is equal to the initial energy 
\begin{equation}
\label{initial energy}
E_0:= \frac{1}{2}\left(\parallel \nabla u_0 \parallel_{L^2(\Omega)}^2 + \parallel u_1 \parallel_{L^2(\Omega)}^2\right)\,.
\end{equation}
\end{corollary}
We conclude the treatment of the weak solution with the following well posedness result, for a proof see \cite{Fourier}.
\begin{theorem}[Well posedness]
\label{t:well posedness}
For each $(w_0,w_1)\in L^2(\Omega)\times H^{-1}(\Omega)$ and $v\in L^2\left( \Sigma \right)$, there exist a unique weak solution of (\ref{control1}) in the sense of Definition \ref{weak def}.
\end{theorem}
For brevity, as we said above, here we confine ourselves to the study of the following equivalent property.
\begin{definition}[Exact observability]
\label{observability}
Let $u$ be the unique mild solution of (\ref{omo}) for the initial data $(u_0,u_1)\in H^1_0(\Omega)\times L^2(\Omega)$. Then the system is called exactly observable in time $T>0$ if there exists a constant $C>0$ such that
\begin{multline}
\label{observability inequality}
\int_{\Sigma} |\partial_{\nu} u|^2\,d\Gamma\otimes dt =\int_{0}^T \int_{\Gamma(t)} |\partial_{\nu} u|^2 \, d\Gamma\,dt \geq \\ 
C \left(
\parallel \nabla u_0 \parallel_{L^2(\Omega)}^2 + \parallel u_1 \parallel_{L^2(\Omega)}^2 \right)\,, 
\end{multline} 
for all $(u_0,u_1)\in H^1_0(\Omega)\times L^2(\Omega)$.
\end{definition}
We remind that $\partial_{\nu} u$ is a well defined $L^2(\partial \Omega \times (0,T))$ element (see Remark \ref{hidden}).\\
We are now in a position to prove an identity, which is the basic tool in proving the exact observability for some $\Sigma$'s.
\begin{lemma}[A multiplier identity]
\label{multiplier lemma}
Let $u$ be the mild solution corresponding to initial data $(u_0,u_1)\in H^1_0(\Omega)\times L^2(\Omega)$. Then for each $0\leq s<\tau$ and $\xi\in \mathbf{R}^d$, we have the following identity
\begin{multline}
\label{identity}
\frac{1}{2}\int_{s}^{\tau}\int_{\partial \Omega} (x-\xi)\cdot \nu \,|\partial_{\nu} u|^2 d \Gamma(x)\,dt =\\
\left[ \int_{\Omega} u_t \left( \nabla u \cdot (x-\xi) + \frac{d-1}{2} u \right)\,dx\right]_{s}^{\tau} +
 (\tau-s)E_0\,;
\end{multline}
where $E_0$ is the initial energy of the system as defined in (\ref{initial energy}) and $[f(t)]_{s'}^{s''}:=f(s'')-f(s')$.
\end{lemma}
\begin{proof}
For reader's convenience, we divide the proof into three steps.\\
\textbf{Step 1} We first suppose that $u_0,u_1$ are smoother, i.e. $u_0\in H^2(\Omega)\cap H^1_0(\Omega)$ and $u_1\in H^1_0(\Omega)$. By Proposition \ref{prop:well posedness omo}, the solution to (\ref{omo}) is a classical solution, in particular the quantities $u_{x_k,x_j}$, for $1\leq k,j \leq d$, belong to $C(\mathbf{R}^+;L^2(\Omega))$; this will be used in the proof.\\
Under this assumption, the wave equation is solved by $u$ almost everywhere on $\Omega\times \mathbf{R}^+$, then we can multiply the wave equation by $(x-\xi)\cdot \nabla u$ and on integrating over $\Omega \times (s,t)$, we obtain
\begin{multline}
\label{eq:int null}
\int_{\Omega}\int_{s}^{\tau} ((x-\xi)\cdot \nabla u) \,u_{tt}\, dx\,dt\\
- \int_{\Omega}\int_{s}^{\tau} ((x-\xi)\cdot \nabla u) \,\Delta u\, dx\,dt
=0\,.
\end{multline}
For the first term on the LHS in (\ref{eq:int null}), with a simple integration by parts argument, we obtain
\begin{align*}
\int_{\Omega}\int_{s}^{\tau} &(x-\xi)\cdot \nabla u\, u_{tt}\,dx\\
&= \left[\int_{\Omega} (x-\xi)\cdot \nabla u\, u_t\,dx\right]_s^{\tau}
 -\int_{\Omega}\int_s^{\tau}(x-\xi)\cdot \frac{1}{2}\nabla ((u_t)^2) dx\,dt \\
&=\left[\int_{\Omega} (x-\xi)\cdot \nabla u\, u_t\,dx\right]_s^{\tau} + \frac{d}{2}\int_{\Omega}\int_s^{\tau} (u_t)^2 \, dx\,dt\,;
\end{align*}
where in the last inequality we have used the Green's formulas, $u_t=0$ on $\partial \Omega\times (s,\tau)$ and  $\nabla \cdot (x-\xi)=d$.\\
For the second term in (\ref{eq:int null}) we use a similar argument. Indeed, by Green's formulas
\begin{align*}
\int_{\Omega}\int_{s}^{\tau}& ((x-\xi)\cdot \nabla u)\cdot \Delta u\,dx\,dt\\
&= \int_{\partial \Omega} \int_s^{\tau} \partial_{\nu} u (x-\xi)\cdot \nabla u \,\,d\Gamma\,dt\\
&- \int_{\Omega} \int_{s}^{\tau} \nabla((x-\xi)\cdot \nabla u)\cdot \nabla u\,dx\,dt\,.
\end{align*}
Now, using that $((u_{x_j})^2)_{x_k}=2u_{x_j,x_k}$ and $\nabla (x_k-\xi_k)_{x_j}=\delta_{k,j}$ for all $1\leq k,j\leq d$ (here $\delta_{k,j}$ is the Kronecker's delta), one obtains
\begin{equation*}
\nabla((x-\xi)\cdot \nabla u)\cdot \nabla u = |\nabla u|^2 +\frac{1}{2} (x-\xi)\cdot \nabla (|\nabla u|^2)\,;
\end{equation*}
pointwise. With simple computations, we have
\begin{align*}
&\int_{\Omega}\int_{s}^{\tau} ((x-\xi)\cdot \nabla u)\cdot \Delta u\,dx\,dt=\\
&\frac{1}{2} \int_{\partial \Omega} \int_s^{\tau} (x-\xi)\cdot \nu |\partial_{\nu}u|^2 + \frac{d-2}{2} \int_{\Omega}\int_{s}^{\tau}  |\nabla u|^2\,dx\,dt \,.
\end{align*}
Putting all together and using that $E(t)\equiv E_0$ (cfr. Corollary \ref{energy cor}), we obtain
\begin{multline}
\label{eq:step 1}
\frac{1}{2} \int_{\partial \Omega} \int_s^{\tau} (x-\xi)\cdot \nu |\partial_{\nu}u|^2 =\left[\int_{\Omega} (x-\xi)\cdot \nabla u\, u_t\,dx\right]_s^{\tau} \\
+ (t-\tau) E_0 + \frac{d-1}{2} \int_{\Omega}\int_s^{\tau} |u_t|^2-|\nabla u|^2\,dx\,dt\,.
\end{multline}
\textbf{Step 2} In this step we rewrite the last term on LHS in (\ref{eq:step 1}). Indeed, under the assumption of \textbf{Step 1}, the wave equation is solved by $u$ a.e. on $\Omega \times (s,\tau)$, so that on multiplying it by $u$, integration over $\Omega \times (s,\tau)$ and using Green's formulas (recall that $u=0$ on $\partial \Omega\times (s,\tau)$)
\begin{multline}
\label{eq:step 2}
0=\int_{\Omega}\int_s^{\tau} u(u_{tt}-\Delta u)\,dx\,dt\\
= \left[ \int_{\Omega} u u_t \,dx\right]_s^{\tau} - \int_{\Omega}\int_s^{\tau}  |u_t|^2-|\nabla u|^2\,dx\,dt\,.
\end{multline}
\textbf{Step 3} Combining the equalities (\ref{eq:step 1})-(\ref{eq:step 2}) we obtain (\ref{identity}) for classical solutions.\\
In the general case, choose sequences such that $\{u_{0,k}\,:\,k\in \mathbf{N}\}\subset H^2(\Omega)\cap H^1_0(\Omega)$ and $\{u_{1,k}\,:\,k\in \mathbf{N}\}\subset H^1_0(\Omega)$, such that
\begin{equation*}
u_{0,k}\rightarrow u_0 \quad\text{in}\;\,H^1_0(\Omega)\,,\qquad
u_{1,k}\rightarrow u_1\quad\text{in}\;\,L^2(\Omega)\,.
\end{equation*}
Then passing to the limit in the identity (\ref{identity}) valid for the solution $u_k$ for initial data $(u_{0,k},u_{1,k})$, one obtain the claim.
\end{proof}

\begin{remark}
The proof of Corollary 3.1 is based on the multiplication of the wave equation against the function $m(x)\cdot \nabla u=(x-\xi)\cdot \nabla u$, which is called \textit{multiplier}, which justifies the name of the identity.
\end{remark}

\section{ALTERNATING OBSERVATION}
\label{s:alternating}
For alternating observation we mean that there exists a partition $0=:t_{-1}<t_{0}<t_{1}<\dots<t_{N-1}<t_{N}=:T$ of the interval $[0,T]$ such that
\begin{equation}
\label{Gamma alternating}
\Gamma(t)\equiv \Gamma_j\,, \qquad \forall t\in(t_{j-1},t_j)\;\;\; j=0\,\dots,N\,;
\end{equation}
where $\Gamma_j\subset \partial \Omega$ is a fixed subset for each $j=0\,\dots,N$.\\
Under the previous hypothesis, by (\ref{Sigma}) we have
\begin{equation}
\label{Sigma alternating}
\Sigma = \bigcup_{j=0}^N \Gamma_j \times (t_{j-1},t_j)\,,
\end{equation}
up to a set of measure $0$.\\
As pointed out in Section \ref{Introduction} the family $\{\Gamma_i\}_{i=0,\dots,N}$ cannot be chosen arbitrarly; we construct this family in a special form (see (\ref{Gamma_i}) below).\\
To do this, let $\{x_i\}_{i=0,\dots,N}$ be an arbitrary family of points in $\mathbf{R}^d$; 
\begin{itemize}
\item For each $i=0,\dots,N$ define 
\begin{equation}
\label{Ri}
R_i=\max\{|x-x_i|\,|\, x\in \overline{\Omega}\}\,,
\end{equation}
and for each $i=0,\dots,N-1$ define
\begin{equation}
\label{Ri+1,i}
R_{i+1,i}= |x_{i+1}-x_i|\,.
\end{equation}
\item For each $i=0,\dots,N$ define
\begin{equation}
\label{Gamma_i}
\Gamma_i = \{x\in \partial \Omega\,|\, (x-x_i)\cdot \nu >0\}\,.
\end{equation}
\end{itemize}
\begin{theorem}[Alternating observability]
\label{observability th}
In the previous notations, for each real number $T$ such that
\begin{equation}
\label{eq:T alternating}
T > R_N + \sum_{i=0}^{N-1} R_{i+1,i} + R_0\,,
\end{equation}
the system is exactly observable in time $T$, in the sense of Definition \ref{observability} for $\Sigma$ as in (\ref{Sigma alternating}).
\end{theorem}
\begin{proof}
For each $i=0,\dots,N$, use the identity (\ref{identity}) for $\xi=x_i$, $s=t_{i-1}$ and $\tau =t_i$. On summing over $i=0,\dots,N$ such identities, we have
\begin{multline}
\label{step1}
\frac{1}{2}\sum_{i=0}^N\int_{t_{i-1}}^{t_{i}}\int_{\Gamma_i} (x-x_i)\cdot \nu \,|\partial_{\nu} u|^2 d \Gamma(x)\,dt =\\ \sum_{i=0}^N
\left[ \int_{\Omega} u_t \left( \nabla u \cdot (x-x_i) + \frac{d-1}{2} u \,dx\right)\right]_{t_{i-1}}^{t_i}
 + T E_0\,,
\end{multline}
since $\sum_{i=0}^N (t_i - t_{i-1})=T$, where we take as above $T=t_{N}$.\\
It will be useful to adopt the following notation
\begin{equation*}
u^s(x):=u(x,s)\,, \qquad u_t^s(x):=u_t(x,s)\,,
\end{equation*}
where $s\in[0,T]$ and $x\in \Omega$.
In the sum on the RHS of (\ref{step1}) some cancellations are possible:
\begin{multline}
\label{step2}
\sum_{i=0}^N
\left[ \int_{\Omega} u_t \left( \nabla u \cdot (x-x_i) + \frac{d-1}{2} u\right) \,dx\right]_{t_{i-1}}^{t_i} =\\
\int_{\Omega} u_t^T \left( \nabla u^T \cdot (x-x_N) + \frac{d-1}{2} u^T\right) \,dx\\
 - \sum_{i=0}^{N-1} \int_{\Omega} u_t^{t_i} \left( \nabla u^{t_i} \cdot [ (x-x_{i+1})-(x-x_{i})]  \right)\,dx \\
 -
\int_{\Omega} u_1\left( \nabla u_0\cdot (x-x_0) + \frac{d-1}{2} u_0\right) \,dx\,.
\end{multline}
For the first and last term on the RHS of (\ref{step2}) we have
\begin{equation}
\label{estimate1}
\Big| \int_{\Omega} u_t^T \left( \nabla u^T\cdot (x-x_N) + \frac{d-1}{2} u^T\right) \,dx\Big| 
\leq R_N E_0\,,
\end{equation}
\begin{equation}
\label{estimate2}
\Big|\int_{\Omega} u_1 \left( \nabla u_0 \cdot (x-x_0) + \frac{d-1}{2} u_0\right) \,dx\Big|
 \leq R_1 E_0\,.
\end{equation}
Moreover, for the terms in the sum on the RHS of (\ref{step2}), we have
\begin{equation}
\label{estimate3}
\Big| \int_{\Omega} u_t^{t_i}\left( \nabla u^{t_i} \cdot (x_i-x_{i+1})\right)\,dx\Big| \leq
 R_{i+1,i} E_0\,,
\end{equation}
for each $i=0,\dots,N$.\\
For convenience, we postpone the proof of the inequalities (\ref{estimate1})-(\ref{estimate3}), see Lemma \ref{estimate} below.\\
By definition of $\Gamma_i$ in (\ref{Gamma_i}) and $R_i$ in (\ref{Ri}), clearly we have
\begin{equation}
\label{estimate4}
R_i\int_{t_{i-1}}^{t_{i}}\int_{\Gamma_i} \,|\partial_{\nu} u|^2 d\Gamma\,dt \geq 
\int_{t_{i-1}}^{t_{i}}\int_{\partial \Omega} (x-x_i)\cdot \nu \,|\partial_{\nu} u|^2 d \Gamma\,dt\,,
\end{equation}
for each $i=0,\dots,N$.\\
Using the equation (\ref{step1}), the identity (\ref{step2}) and the estimates (\ref{estimate1})-(\ref{estimate4}), we have 
\begin{multline}
\label{step3}
\left( \max_{i=0,\dots,N} R_i\right)\sum_{i=0}^N \int_{t_{i-1}}^{t_{i}}\int_{\Gamma_i} \,|\partial_{\nu} u|^2 d \Gamma\,dt 
\geq \\
2(T - R_N - R_{N,N-1}- \dots - R_{1,0}-R_0)E_0\,.
\end{multline}
Now, by assumption $T> R_N + \sum_{i=1}^{N-1} R_{i+1,i} + R_1$, so there exists a constant $C=C_T>0$, which depends on $T$, such that
\begin{equation}
\label{step4}
\sum_{i=0}^N \int_{t_{i-1}}^{t_{i}}\int_{\Gamma_i} |\partial_{\nu} u|^2 d \Gamma\,dt \geq C_T E_0\,,
\end{equation}
which is exactly the observability inequality, as defined in Definition \ref{observability}.
\end{proof}
We now prove the estimates (\ref{estimate1})-(\ref{estimate3}); in particular (\ref{estimate1})-(\ref{estimate2}) and (\ref{estimate3}) follow respectively by $i)$ and $ii)$ of the following Lemma.
\begin{lemma}
\label{estimate}
For each mild solution $u\in C(\R^+;H^1_0(\Omega))\cap C^1(\R^+;L^2(\Omega))$ of (\ref{omo}), the following holds.
\begin{itemize}
\item[i)] For each $\xi\in \mathbf{R}^d$ and $s\in \mathbf{R}$, then
\begin{equation*}
\Big| \int_{\Omega} u_t^s \left( \nabla u^s\cdot (x-\xi) + \frac{d-1}{2} u^s\right) \,dx\Big| \\
\leq R_{\xi} E_0\,,
\end{equation*}
where $R_{\xi}=\max\{|x-\xi|\,:\,x\in \overline{\Omega}\}$.
\item[ii)] For each $\xi,\eta\in \mathbf{R}^d$ and $s\in \mathbf{R}^+$, then
\begin{equation*}
 \Big| \int_{\Omega}  u_t^s\left( \nabla u^s\cdot(\xi -\eta)\right)dx \Big|\leq R_{\xi,\eta}E_0\,,
\end{equation*}
where $R_{\xi,\eta}=|\xi- \eta|$.
\end{itemize}
\end{lemma}
\begin{proof}
$i)$ By Chauchy-Schwarz inequality, we have
\begin{multline}
\label{l:parte i)}
\Big|\int_{\Omega} u_t^s \left( \nabla u^s\cdot (x-\xi) + \frac{d-1}{2} u^s\right) \,dx\Big|  \\
\leq\| u_t^s\|_{L^2} \|\nabla u^s\cdot (\cdot-\xi) + \frac{d-1}{2} u^s\|_{L^2}\\
\leq\frac{R_{\xi}}{2} \| u_t^s\|_{L^2}^2 + \frac{1}{2\,R_{\xi}}\|\nabla u^s\cdot (\cdot-\xi) + \frac{d-1}{2} u^s\|_{L^2}^2\,;
\end{multline}
here $L^2:=L^2(\Omega)$. Note that, 
\begin{align*}
\|&\nabla u^s\cdot (\cdot-\xi) + \frac{d-1}{2} u^s\|_{L^2}^2 \\
=&\|\nabla u^s\cdot (\cdot-\xi) \|^2_{L^2} + (d-1)(\nabla u^s\cdot (\cdot-\xi) ,u^s)_{L^2}\\
&+ \frac{(d-1)^2}{4}\|u^s\|_{L^2}^2\,.
\end{align*}
Since $u=0$ on $\partial \Omega$, then using Green's identity the middle term in the RHS of the previous equation is equal to
\begin{align*}
 \int_{\Omega}& (\nabla u^s\,\, u^s)\cdot(x-\xi)\,d\,x\\
&=\frac{1}{2}\int_{\Omega} (\nabla (u^s)^2)\cdot (x-\xi)\,\,d\,x\\
&=-\frac{1}{2}\int_{\Omega} (u^s)^2( \nabla \cdot(x-\xi))\,d\,x=-\frac{d}{2}\|u^s \|_{L^2}^2\,.
\end{align*}
This implies
\begin{align*}
\|\nabla u^s\cdot & (\cdot-\xi) + \frac{d-1}{2} u^s\|_{L^2}^2 = \\
&\|\nabla u^s\cdot (\cdot-\xi) \|^2_{L^2} \\
&+ \left[ -(d-1)\frac{d}{2} + \frac{(d-1)^2}{4}\right]\|u^s\|_{L^2}^2\\
&\leq \|\nabla u^s\cdot (\cdot-\xi) \|^2_{L^2}\leq R_{\xi}^2 \|\nabla u^s\|_{L^2}^2\,,
\end{align*}
since $$-\frac{d(d-1)}{2} + \frac{(d-1)^2}{4} = -\frac{d^2}{2}+\frac{1}{4}<0\,,$$
 for all $d\geq 1$.\\
Now returning to (\ref{l:parte i)}), we have
\begin{align*}
\Big|\int_{\Omega}& u_t^s \left( \nabla u^s\cdot (x-\xi) + \frac{d-1}{2} u^s\right) \,dx\Big| \\
&\leq\frac{R_{\xi}}{2} \| u_t^s\|_{L^2}^2 + \frac{1}{2\,R_{\xi}} R_{\xi}^2 \|\nabla u^s\|_{L^2}^2\\
&=\frac{R_{\xi}}{2} (\| u_t^s\|_{L^2}^2+\|\nabla u^s\|_{L^2}^2)=R_{\xi}E_0\,;
\end{align*}
where the last inequality follows by the energy conservation (cfr. Corollary \ref{energy cor}).\\
$ii)$ The proof the second part of the Lemma is easier. Indeed, by Cauchy-Schwarz inequality, we have
\begin{align*}
 \Big| \int_{\Omega}  u_t^s &\nabla u^s\cdot( \xi-\eta) \;dx \Big|\\
&\leq \parallel  u_t^s\parallel_{L^2} \parallel\nabla u^s\cdot(\xi-\eta) \parallel_{L^2}\\
&\leq R_{\xi,\eta} \parallel  u_t^s\parallel_{L^2} \parallel\nabla u^s \parallel_{L^2} \leq R_{\xi,\eta} E_0\,;
\end{align*}
where the last inequality follows another time by the energy conservation.
\end{proof}

\subsection{The role of $\{t_i\}_{i=-1,\dots,N}$.}
\label{theroleof}
One may wonder what the role is of the family $\{t_i\}_{i=-1,\dots,N}$, since in Theorem \ref{observability th} only the sum $T=\sum_{i=0}^{N}(t_{i}-t_{i-1})$ appears.\\
To explain the role of these values, we have to recall that if $N=0$ (i.e. fixed support control) then Theorem \ref{observability th} implies that the exact controllability holds for any $T> 2\max \{|x-x_0|\,|\,x\in \overline{\Omega}\}$, where $x_0\in \mathbf{R}^d$ is a fixed point. So if for an index $j\in\{0,\dots,N\}$ we have 
\begin{equation}
\label{role1}
|t_{j}-t_{{j}-1}| > 2 \max\{|x-x_{j}|\,|\,x\in \overline{\Omega}\}\,,
\end{equation}
then we can construct a control $v$ such that $supp\,\, v \subset \overline{\Gamma_j} \times [t_{j-1},t_j]$ by using the fixed support case of Theorem \ref{observability th}. Indeed, 
fix $(w_0,w_1)\in L^2(\Omega)\times H^{-1}(\Omega)$ and let be $w$ the unique weak solution of the problem (\ref{control1}) for $T=t_j$ and $v=0$; so by Theorem \ref{t:well posedness} we have that $(w(t_{j-1}),w_t(t_{j-1}))$ is well defined as an element of $L^2(\Omega)\times H^{-1}(\Omega)$. Since the inequality in (\ref{role1}) holds, the fixed case of Theorem \ref{observability th} provides a control $\tilde{v}$ such that the solution at time $t=t_j$ satisfies the null condition; so after defining $v=0$ for $t<t_{j-1}$ and $v=\tilde{v}$ for $t_{j-1}<t<t_j$ we have constructed a control for the initial condition $(w_0,w_1)$. So if the inequality (\ref{role1}) holds for an index $j$ the controllability results follows by classical results, (see for istance Theorem 6.1, Chapter 1 of \cite{Lions}). In the following section we produce an explicit example in which the inequality (\ref{role1}) is not satisfied by any $i=0,\dots,N$; so our investigation produces new results on controllability for the wave equation.\\
By the way, we point out that Theorem \ref{observability th} allows us to apply the Hilbert uniqueness method (or HUM, see \cite{Fourier}-\cite{Lions}); with some effort one can prove that the control provided by this method minimizes the \textit{energy} of the control over the possible controls (see Chapter 7 of \cite{Lions}). For this reason even in the case when (\ref{role1}) holds our Theorem yields the existence of a minimizer control; this remarkable property of the HUM control can be useful in applications. For brevity we do not reproduce the needed calculations.

\section{VARIABLE OBSERVATION}
In this section we recall a fairly general observability Theorem in which the subset $\Gamma(t)$ of observation at time $t$ can vary at each time $t\in(0,T)$.\\
As explained in Section \ref{Introduction}, the family $\{\Gamma(t)\}_{t\in(0,T)}$ cannot be arbitrary and will be constructed in a similar fashion to $\{\Gamma_i\}_{i=0,\dots,N}$ (see (\ref{Gamma_i}) in Section \ref{s:alternating}).\\
To do this, let $\varphi:[0,T]\rightarrow \mathbf{R}^d$ be a continuous and piecewise differentiable curve in $\mathbf{R}^d$ of finite length, i.e.
\begin{equation}
\label{length finite}
L(\varphi):=\int_0^T |\varphi'(t)|\,dt< +\infty\,.
\end{equation}
Furthermore, define
\begin{align}
\label{Gamma_t}
\Gamma_{\varphi}(t)&=\{x\in \partial \Omega\,|\,(x-\varphi(t))\cdot \nu >0\}\,,\\
\label{Sigma phi}
\Sigma_{\varphi} &= \bigcup_{t\in(0,T)} \Gamma_{\varphi}(t)\times \{t\}\,,\\
\label{ci}
c_i&=\max_{\overline{\Omega}}|x-\varphi(i)|\,, \qquad i=0,T\,.
\end{align}
Hereafter we assume that $\Sigma_{\varphi}$ is $\mathcal{H}^{d-1}\otimes \mathcal{L}^1$-measurable; for further discussion on this topic see \cite{Preprint}.\\
Also, let $P={0=t_{-1}<t_{0}<\dots t_{N-1}<t_N=T}$ be a partition of the interval $[0,T]$; define 
\begin{align}
\label{xphij}
x_{\varphi,j}&=\varphi(t_{j-1})\,,\\
\label{Gammaj}
\Gamma_{\varphi,j} &=\{x\in \partial \Omega\,|\,(x-x_{\varphi,j})\cdot \nu >0\}\,,
\end{align} 
for all $j=0,\dots,N$. For future convenience, we set 
\begin{equation}
\label{Sigmaphij}
\Sigma_{\varphi}^P:=\bigcup_{j=0}^{N} \Gamma_{\varphi,j} \times (t_{j-1},t_{j})\,.
\end{equation}
In next theorem, we have to consider a sequence of partitions $\{P_k\}_{k\in \mathbf{N}}$ of the interval $[0,T]$, so we will add the upper index $k$ in (\ref{xphij})-(\ref{Gammaj}) in order to keep trace of the dependence on $P_k$; and we will set $\Sigma_{\varphi}^k:= \Sigma_{\varphi}^{P_k}$.\\
Now we are ready to state the main result of this section.
\begin{theorem}[Variable Support Observability]
\label{variable corollary}
Under the above hypothesis, suppose there exists a sequence of partitions $\{P_k\}_{k\in\mathbf{N}}$ of the interval $[0,T]$, such that $\sup_{\{t_{j-1}^k,t_j^k\}\in P_k}|t_j^k-t_{j-1}^k| \searrow 0$ as $k\nearrow \infty$ and
\begin{equation}
\label{convergence}
\lim_{k\rightarrow\infty}\left(\mathcal{H}^{d-1}\otimes \mathcal{L}^1\right)\left(\Sigma_{\varphi} \Delta \Sigma_{\varphi}^k \right)= 0\,,
\end{equation}
where $P_k =\{0=t_{-1}^k<t_0^k<\dots<t_N^k=T\}$, $\mathcal{L}^1$ is the Lebesgue measure on $\mathbf{R}^1$ and $A\Delta B =(A\setminus B) \cup (B\setminus A)$.\\
Furthermore, suppose that $T$ verifies
\begin{equation}
\label{Tvariable}
T > c_0 + \int_{0}^T |\varphi'(t)|\,dt+c_N\,.
\end{equation}
Then the system is exactly observable in time $T$ (see Definition \ref{observability}) for $\Sigma =\Sigma_{\varphi}$.
\end{theorem}
For brevity we do not report the proof of Theorem \ref{variable corollary}; for the proof, other applications and further extension of this result we refer to \cite{Preprint}.

\section{APPLICATIONS}
In this section we give some applications of Theorems \ref{observability th} and \ref{variable corollary} in order to show the potentiality of these results.\\ 
To begin, we focus our attention to $\Omega=\B:=\{x\in \mathbf{R}^2\,|\,|x|<1\}$, i.e. the ball with center $0$ and radius $1$. As usual, we denote with $\S:=\partial \B$ the unit circle with center $0 $. 
%\begin{proof}
%It is an easy consequence of Corollary \ref{control corollary} with $N=1$ and the choice $x_0=(1,1)$ and $x_1=(0,0)$. It is clear that $\Gamma_0=l_0$ or $\Gamma_1=l_1$ by definition in (\ref{Gamma_i}).
%\end{proof}
%\begin{remark}
%\label{remarksquare}
%Note that in Corollary \ref{cor square 1} there is no assumption on $t_0$, since by our notation $t_{-1}=0$ and $t_1=T$. Due to the discussion in the subsection \ref{theroleof}, it is clear that the Corollary does not follow trivially from known result if 
%\begin{equation}
%|t_0|\leq 2 \sqrt{2}\,.
%\end{equation}
%So we may require that $t_0\in(0,2\sqrt{2})$ to avoid this situation.
%\end{remark}

\begin{corollary}[$1$-time alternating - Circle case]
\label{1circle cor}
Let $T>2(1+\sqrt{2})$ be a real number, define
\begin{align}
\label{d0}
d_0 &=\Big\{ z\in \S\,\Big|\, \frac{\pi}{2}< arg\,z <2\pi \Big\},\\
\label{d1}
d_1 &=\Big\{z\in\S \,\Big|\, 0< arg\,z <\frac{3\pi}{2} \Big\}\,.
\end{align}
Then the system (\ref{control1}) is exactly controllable for $\Omega=\B$, $t_0$ is an arbitrary element of $(0,T)$, and
\begin{equation*}
\Sigma = d_0 \times (0,t_0) \bigcup d_1 \times (t_0,T)\,.
\end{equation*}
\end{corollary}

\begin{proof}
It is an easy consequence of Theorem \ref{observability th} with the choice $x_0=(1,1)$ and $x_1=(1,-1)$. Indeed, by (\ref{Ri})-(\ref{Ri+1,i}) it is clear that
\begin{equation*}
R_0=R_1= \sqrt{2}+1\,, \qquad R_{0,1}=2\sqrt{2}\,.
\end{equation*}
Moreover, the condition (\ref{eq:T alternating}) in Theorem \ref{observability th}, implies $T>2(1+\sqrt{2})$. Lastly, with simple geometrical consideration, one can see $\Gamma_0=d_0$ and $\Gamma_1=d_1$ by (\ref{Gamma_i}). This concludes the proof.
\end{proof}

\begin{remark}
\label{1remark}
Note that in Corollary \ref{1circle cor} there is no assumption on the value $t_0\in(0,T)$; due to the discussion in subsection \ref{theroleof} it is clear that Corollary \ref{1circle cor} does not follow trivially by known results if
\begin{equation}
\label{condition t_0}
|t_0|< 2(1+\sqrt{2})\,.
\end{equation}
\end{remark}
We now extend Corollary \ref{1circle cor} to the $N$-times alternating case.

\begin{corollary}[$N$-times alternating - Circle case]
\label{Ncircle cor}
Let $T>2(N+1)+2\sqrt{2}$ be a real number and let $d_0,d_1$ be as in (\ref{d0})-(\ref{d1}). Then the system (\ref{control1}) is exactly controllable for $\Omega=\B$, $\{t_i\}_{i=0,\dots,N-1}$ any increasing finite subfamily of $(0,T)$, and
\begin{equation}
\label{Sigma:N alternating circle}
\Sigma =\left( \bigcup_{i=0,\,i\in 2\mathbf{N}}^N d_0 \times (t_{j-1},t_j)\right)\cup
\left( \bigcup_{i=0,\,i\in 2\mathbf{N}+1}^N d_1 \times (t_{j-1},t_j)\right)\,.
\end{equation}
\end{corollary}
\begin{proof}
It is similar to the proof of Corollary \ref{1circle cor}. In this case, we have to choose $x_i\equiv x_0=(1,1)$ if $i$ is even or $x_i\equiv x_1=(1,-1)$. Similar to Corollary \ref{1circle cor}, we have
\begin{equation*}
R_0=R_1= \sqrt{2}+1\,, \qquad R_{i,i+1}=2\sqrt{2}\,,
\end{equation*}
for all $i=0,\dots,N-1$. As in the proof of Corollary \ref{1circle cor}, by (\ref{Gamma_i}) we have that $\Gamma_i\equiv d_0$ if $i$ is even, otherwise $\Gamma_i\equiv d_0$ and $\Sigma$ is as in (\ref{Sigma:N alternating circle}).
\end{proof}
\begin{remark}
\label{r:2}
As did in Remark \ref{1remark} for Corollary \ref{1circle cor}, we observe that in Corollary \ref{Ncircle cor} there is no assumption on the values $t_{i}$ for $i=0,\dots,N-1$. In this case we may require that
\begin{equation}
\label{condition t_j}
|t_j-t_{j-1}|< 2(1+\sqrt{2})\,, \qquad \forall j=0,\dots,N\,.
\end{equation}
\end{remark}
Below we give an interesting application of Theorem \ref{variable corollary}.
\begin{corollary}
\label{variable circle}
Let $\alpha\in \mathbf{R}^+$ be a positive real number, such that
\begin{equation}
\label{alpha bound}
\alpha < \frac{\pi}{4(1+\sqrt{2})+\pi \sqrt{2}}\,.
\end{equation}
Then the system (\ref{control1}) is exactly controllable in time $T=\pi/(2\alpha)$ with $\Omega=\B$ and $\Sigma=\cup_{t\in(0,T)} \Gamma(t)\times \{t\}$; where
\begin{equation}
\Gamma(t)= \Big\{ z\in \S\,\Big|\, \frac{\pi}{4}+\alpha\,t< arg\,z < \frac{7\pi}{4}+\alpha\,t\Big\}\,.
\end{equation}
\begin{proof}
Set $\varphi(t)=\sqrt{2}(\cos(\alpha t),\sin(\alpha t))$ for $t\in (0,\pi/(2\alpha))$.  Since it is easy to check that $\Gamma(t)=\Gamma_{\varphi}(t)$ (recall that $\Gamma_{\varphi}(t)$ is defined in (\ref{Gamma_t})), to prove the corollary we have only to check the hypothesis of Theorem \ref{variable corollary}.\\
Indeed, it is clear that the condition (\ref{convergence}) holds by taking a sequence of partition $P_k=\{t_j^k\}_{j=0,\dots,k}$ with $t_j^k=(j/k)T$ and $T:=\pi/(2\alpha)$.\\
By construction, the length of the curve is $L(\varphi)=\sqrt{2}(\pi/2)$, and $c_{i}=1+\sqrt{2}$ for $i=0,T$. So condition (\ref{Tvariable}) is satisfied if 
\begin{equation}
\frac{\pi}{2	\alpha} > 2(1+\sqrt{2})+\sqrt{2}\frac{\pi}{2}\,,
\end{equation}
which is equivalent to (\ref{alpha bound}). 
\end{proof}
\end{corollary}

We now analyse a case of a convex domain; i.e. the interior of the hexagon with side $1$ and center $0$, it will be denoted by $\H\subset \R^2$ and we set $\E:=\partial \H$.\\
In this situation, we have analogous result to Corollaries \ref{1circle cor}, \ref{Ncircle cor} and \ref{variable circle}:
\begin{corollary}[$1$-time alternating - Hexagon case I]
\label{1hexagon cor}
Let $T>6$ be a real number, define
\begin{align}
\label{e0}
e_0 &=\Big\{ (x,y)\in \E\,\Big|\,x<\frac{1}{2}\Big\},\\
\label{e1}
e_1 &=\Big\{(x,y)\in\E \,\Big|\, x>-\frac{1}{2}\Big\}\,.
\end{align}
Then the system (\ref{control1}) is exactly controllable for $\Omega=\H$, $t_0$ is an arbitrary element of $(0,T)$, and
\begin{equation*}
\Sigma = e_0 \times (0,t_0) \bigcup e_1 \times (t_0,T)\,.
\end{equation*}
\end{corollary}
\begin{proof}
As in the Proof of Corollary \ref{1circle cor}, we use Theorem \ref{observability th} with the choice $x_0=(1,0)$ and $x_1=(-1,0)$; it is also clear that $\Gamma_0=e_0$ and $\Gamma_1=e_1$ by definition (\ref{Gamma_i}). \\
Furthermore, by (\ref{Ri})-(\ref{Ri+1,i}), we have
\begin{equation*}
R_0=R_1= R_{0,1}=2\,;
\end{equation*}
then by (\ref{eq:T alternating}) the claim follows.
\end{proof}

\begin{corollary}[$N$-times alternating - Hexagon case]
\label{Nhexagon cor}
Let $T>2(2+N)$ be a real number and let $e_0,e_1$ be as in (\ref{e0})-(\ref{e1}). Then the system (\ref{control1}) is exactly controllable for $\Omega=\H$, $\{t_i\}_{i=0,\dots,N-1}$ any increasing finite subfamily of $(0,T)$, and
\begin{equation*}
\Sigma =\left( \bigcup_{i=0,\,i\in 2\mathbf{N}}^N e_0 \times (t_{j-1},t_j)\right)\cup
\left( \bigcup_{i=0,\,i\in 2\mathbf{N}+1}^N e_1 \times (t_{j-1},t_j)\right)\,.
\end{equation*}
\end{corollary}
\begin{proof}
Is similar to the proof of Corollaries \ref{Ncircle cor} and \ref{1hexagon cor}.\\
In this case, we have to choose $x_i\equiv x_0=(1,0)$ if $i$ is even or $x_i\equiv x_1=(-1,0)$ otherwise. Moreover, by (\ref{Ri})-(\ref{Ri+1,i}) we obtain $$R_0=R_N=R_{i+1,i}=2\,.$$
Then condition (\ref{eq:T alternating}) is equivalent to $T>4+2N=2(2+N)$ and the claim follows.
\end{proof}

The last application consists in the following:
\begin{corollary}[$1$-time alternating - Hexagon case II]
\label{1hexagon corII}
Let $T>5 \sqrt{3}$ be a real number, define
\begin{align*}
%\label{e'0}
e'_0 &=\Big\{ (x,y)\in \E\,\Big|\,y<\frac{\sqrt{3}}{2}x\Big\},\\
%\label{e'1}
e'_1 &=\Big\{(x,y)\in\E \,\Big|\, y>\frac{\sqrt{3}}{2}x\Big\}\,.
\end{align*}
Then the system (\ref{control1}) is exactly controllable for $\Omega=\H$, $t_0$ is an arbitrary element of $(0,T)$, and
\begin{equation*}
\Sigma = e'_0 \times (0,t_0) \bigcup e'_1 \times (t_0,T)\,.
\end{equation*}
\end{corollary}

\begin{proof}
As in the Proof of Corollary \ref{1hexagon cor}, we use Theorem \ref{observability th} with the choice
\begin{equation*}
x_0= \left( \frac{3}{2},\frac{\sqrt{3}}{2}\right)\,, \qquad x_1=\left( -\frac{3}{2},-\frac{\sqrt{3}}{2} \right)\,.
\end{equation*}
By (\ref{Ri})-(\ref{Ri+1,i}), we have
\begin{equation*}
R_0=R_1=\frac{3\sqrt{3}}{2}\,, \qquad R_{0,1}=2\sqrt{3}\,.
\end{equation*}
In this case, the condition (\ref{eq:T alternating}) is $T>2(3\sqrt{3})/2+2\sqrt{3}=5\sqrt{3}$ as stated in the Corollary.\\
Furthermore, with this choice of $x_0,x_1$, by (\ref{Gamma_i}) one can prove that $\Gamma_0=e'_0$ and $\Gamma_1=e'_1$; then the claim follows.
\end{proof}

The analogous of Corollary \ref{Nhexagon cor} for the choice of $x_0,x_1$ made in Corollary \ref{1hexagon corII} can be easily proven by the same argumentations (in this case $T>\sqrt{3}(3 + 2N)$). In order to avoid repetition we do not record the proof here; the same considerations are valid for the content of Remarks \ref{1remark} and \ref{r:2}.\\
\\
We point out that Theorems \ref{observability th} and \ref{variable corollary} can be applied to a very wide range of domains in $\mathbf{R}^d$ for any $d>1$; we have chosen only two cases in this section in order to give a flavour of the possible applications.\\
Theorem \ref{variable corollary} needs the additional condition (\ref{convergence}), which is not in general immediate to prove, though it can be proven to be valid for smooth $\Omega$ and $\varphi$; this can be useful in applications.

\section{CONCLUSIONS}
\label{s:conclusion}
%questo lavoro mette in risalto la variabilità del supporto del controllo che non era mai stato considerato precedentemente.... potrebbe essere un futuro indirizzo del controllo di equazioni lineari.
To conclude we have shown that the exact controllability holds if $T$ is sufficiently large and the subset of observation $\Sigma$ is suitable.\\
More specifically:
\begin{itemize}
\item Let $\{x_i\}_{i=0\,\dots,N}$ be an arbitrary family of points in $\mathbf{R}^d$, $\{t_i\}_{i=-1,\dots,N}$ an increasing family of positive numbers such that $t_{-1}=0$ and $T:=t_{N}$; then the system (\ref{control1}) is exactly controllable if
\begin{equation*}
T> R_0 + \sum_{i=0}^{N-1} R_{i+1,i} + R_{N}\,;
\end{equation*}
where $R_i,R_{i+1,i}$ are defined in (\ref{Ri})-(\ref{Ri+1,i}) and
\begin{align*}
\Sigma &= \bigcup_{i=0}^N \Gamma_{i} \times (t_i,t_{i-1})\,,\\
\Gamma_i :&= \{x\in \partial \Omega \,|\, (x-x_i)\cdot \nu >0\}\,.
\end{align*}
See Theorem \ref{observability th}.
\item Let $\varphi:[0,T]\rightarrow \mathbf{R}^d$ be a continuous and piecewise differentiable curve of finite length (see (\ref{length finite})). If
\begin{align*}
T&> c_0 + \int_0^T |\varphi'(t)|\,dt + c_N\,;\\
c_i :&= \max_{x\in  \overline{\Omega}} |x-\varphi(i)|\,, \qquad i=0,T\,,
\end{align*}
the system (\ref{control1}) is exactly controllable with
\begin{align*}
\Sigma &= \bigcup_{t\in(0,T)} \Gamma_{\varphi}(t) \times\{t\}\,,\\
\Gamma_{\varphi}(t) :&= \{x\in \partial \Omega \,|\, (x-\varphi(t))\cdot \nu >0\}\,.
\end{align*}
See Theorem \ref{variable corollary}.
\end{itemize}
These results can be helpful in engineering applications in which one has to control the evolution of a structure (whose dynamics is governed by the wave equation) by means of an action on a portion of the boundary (see \cite{Carcaterra2}), but for structural reasons, one cannot act for a long time on the same portion of the boundary,  so that the \textit{switch} of the control is necessary.

\end{document}